\documentclass{amsproc}
\numberwithin{equation}{section}
\usepackage{hyperref}
\hypersetup{nesting=true,debug=true,naturalnames=true}
\usepackage{graphicx,amssymb,upref}
\usepackage[usenames,dvipsnames]{pstricks}
\usepackage{epsfig}
\usepackage{color}
\usepackage[utf8]{inputenc}
\usepackage[T1]{fontenc}
\usepackage{amsmath, amsthm}
\usepackage[english]{babel}
\usepackage{amssymb}
\usepackage{latexsym}
\usepackage{graphicx,amssymb,upref,amsfonts}
\usepackage[all]{xy}
\usepackage{tikz}
\usetikzlibrary{matrix} 
\usepackage{mathtools}
\usepackage{pdfpages}
\usepackage{multirow}
\usepackage{comment}
\usepackage{float}
\usepackage{appendix}
\allowdisplaybreaks

\newcounter{zlist}

  \newcounter{blist} 
  \newenvironment{blist}{\begin{list}{(\alph{blist})}{ 
  \usecounter{blist}\leftmargin2.5em\labelwidth2em\labelsep0.5em 
  \topsep0.6ex 
  \parsep0.3ex plus0.2ex minus0.1ex}}{\end{list}} 

  \newcounter{rlist} 
  \newenvironment{rlist}{\begin{list}{(\roman{rlist})}{ 
  \usecounter{rlist}\leftmargin2.5em\labelwidth2em\labelsep0.5em 
  \topsep0.6ex 
  \parsep0.3ex plus0.2ex minus0.1ex}}{\end{list}}

\date{\today}

\newtheorem{theorem}{Theorem}[section]

\newtheorem{proposition}[theorem]{Proposition}
\newtheorem{corollary}[theorem]{Corollary}
\theoremstyle{definition}
\newtheorem{definition}[theorem]{Definition}
\newtheorem{example}[theorem]{Example}

\newtheorem{remark}[theorem]{Remark}

\newcommand{\bs}{\boldsymbol{\sigma}}
\newcommand{\be}{\boldsymbol{\varepsilon}}

\usepackage[style=iso]{datetime2} 
\usepackage{tikz-cd}
\usepackage[inline]{enumitem} 

\hypersetup{hidelinks}
\numberwithin{equation}{section}

\begin{document}

\title[Ore meets homothetic extensions]{Ore meets homothetic extensions}
\author[Brzezi\'nski]{Tomasz Brzezi\'nski}

\address{Department of Mathematics, Swansea University,
Fabian Way,
  Swansea SA1 8EN, U.K.\ \newline \indent
Faculty of Mathematics, University of Bia{\l}ystok, K.\ Cio{\l}kowskiego  1M,
15-245 Bia\-{\l}ys\-tok, Poland}

\email{T.Brzezinski@swansea.ac.uk}
\author[West]{A.T.M.\ West}
\address{Department of Mathematics, Swansea University,
Fabian Way,
  Swansea SA1 8EN, U.K.}

\email{a.t.m.west@swansea.ac.uk}
\begin{abstract}
    Sufficient and necessary conditions for an extension of a skew-derivation $(\delta_R,\alpha_R)$ of an associative $\mathbb{F}$-algebra $R$ to a skew derivation $(\delta_S,\alpha_S)$ on an extension $S$ of $R$ by $\mathbb{F}$   or a {\em homothetic extension $S$ of $R$ by $\mathbb{F}$} are derived. It is then shown that this yields the unique extension of  the Ore extension $R[x;\alpha_R,\delta_R]$ of $R$  by $\mathbb{F}$ that embeds in the Ore extension $S[x;\alpha_S,\delta_S]$ of $S$ by the extended skew-derivation.
\end{abstract}
\maketitle

\section{Introduction}
Throughout the text, $\mathbb{F}$ denotes a field or ring of integers $\mathbb{Z}$ and $R$ is an associative (but not necessarily unital or commutative) $\mathbb{F}$-algebra (so in the case of an algebra over $\mathbb{Z}$ it is simply an associative ring).  We consider an exact sequence of algebra homomorphisms
\begin{equation}\label{extension}
  \begin{tikzcd}
      0 \arrow[r] & R \arrow[r,"\iota"] & S \arrow[r,"\pi"] & \mathbb{F}\ar[r] &0.
  \end{tikzcd}  
\end{equation}
In other words, $S$  is an algebra extension of $R$ by $\mathbb{F}$.\footnote{In some contexts, in particular in that of homological algebra, one would rather say  that $S$ is an extension of $\mathbb{F}$ by $R$ in this situation. We use the terminology of \cite{Red:ver}.} Necessarily, $\iota(R)$ is an ideal in $S$ and we will identify $R$ with $\iota(R)$. As shown in \cite{AndBrzRyb:ext}, such extensions provide an equivalent (but purely ring-theoretic) description of abelian heaps \cite{Pru:the} with a distributing, associative multiplication known as {\em trusses} \cite{Brz:tru}, \cite{Brz:par} that bridges rings with (two-sided) {\em braces} \cite{Rum:bra}, \cite{CedJes:bra}. Our first aim is to address the following problem. Let $\alpha_R: R\to R$ be an algebra endomorphism and let $\delta_R$ be an $\alpha_R$-derivation or a skew derivation on $R$, that is a linear endomorphism of $R$ such that, for all $a,b\in R$,
\begin{equation}\label{skew.Leibniz}
    \delta_R(ab) = \delta_R(a)b + \alpha_R(a)\delta_R(b),
\end{equation}
(we encode these data as $(\alpha_R,\delta_R)$).
Determine when it is possible to extend $\alpha_R$ to an algebra endomorphism $\alpha_S$ of $S$ and $\delta_R$ to an $\alpha_S$-derivation $\delta_S$ of $S$ such that
\begin{equation}\label{deriv.ext}
 \alpha_S \iota = \iota \alpha_R \quad \mbox{and} \quad 
 \delta_S\iota = \iota \delta_R.
\end{equation}
We address this problem in Section~\ref{sec.deriv} and  solve it completely in Proposition~\ref{prop.alpha.ext} and Proposition~\ref{prop.deriv.ext}. In particular, we note that the extensions of $(\alpha_R,\delta_R)$ split into two distinct classes characterised by the scalar $\varsigma$ equal either to 0 or 1 (these are necessary for the zero endomorphism to be extended to the zero endomorphism in the former case and for the identity to be extended to the identity in the latter). We illustrate the solution of the extension problem by a number of examples. 

Once this first aim is achieved, we proceed to discuss Ore extensions of $S$ determined by extended skew derivations in Section~\ref{sec.Ore}. In particular in Theorem~\ref{thm.data.ext} and Proposition~\ref{prop.Ore-hom} we show that if $(\alpha_R,\delta_R)$ can be extended to a skew-derivation $(\alpha_S,\delta_S)$ with $\varsigma=1$, then there is a (unique) extension of $\mathbb{F}$ by the Ore extension $R[x;\alpha_R,\delta_R]$ of $R$ which embeds in the Ore extension $S[x;\alpha_S,\delta_S]$ yielding the following commutative  diagram of algebra maps with exact rows:
\begin{equation}\label{extension.2}
  \begin{tikzcd}
      0 \arrow[r] & R[x;\alpha_R,\delta_R] \arrow[r,"\tilde\iota"]\arrow[d,equal] & \tilde{S} \arrow[r,"\tilde\pi"] \arrow[hookrightarrow]{d} & \mathbb{F}\ar[r]\arrow[hookrightarrow]{d} &0\\
      0 \arrow[r] & R[x;\alpha_R,\delta_R] \arrow[r,"\iota'"] &  S[x;\alpha_S,\delta_S]\arrow[r,"\pi'"] & \mathbb{F}[x]\ar[r] &0,
  \end{tikzcd}  
\end{equation}
where the last vertical map is the inclusion of scalars.

\section{Extending skew derivations}\label{sec.deriv}

Taking inspiration from the groundbreaking works of Everett \cite{Eve:ext}, Redei \cite{Red:ver}, Mac Lane \cite{Mac:ext} and Petrich \cite{Pet:ide}, it was explained in \cite{AndBrzRyb:ext} that extensions \eqref{extension} can be equivalently and more explicitly described in terms of {\em homothetic data} which we define presently.

\begin{definition}[Double homothetism]\label{def.hom.datum}
    Let $\sigma=(\overleftarrow \sigma,\overrightarrow\sigma)$ be a double operator on $R$, that is, a pair of linear endomorphisms  
    $$
        \overrightarrow\sigma: R\rightarrow R,\quad a\mapsto\sigma a,\qquad \overleftarrow \sigma:R\rightarrow R,\quad a\mapsto a\sigma.
    $$ 
    Then $\sigma$ is a \textit{bimultiplication} \cite{Mac:ext} if, for all $a,b\in R$,
    \begin{equation}\label{eq.bimultiplication}
        \sigma(ab)=(\sigma a)b, \quad (ab)\sigma=a(b\sigma), \quad \text{and}\quad  a(\sigma b) = (a\sigma)b.
    \end{equation}
    A bimultiplication $\sigma$ is a \textit{double homothetism} \cite{Red:ver} if, for all $a\in R$,
    \begin{equation}
        \sigma(a\sigma)=(\sigma a)\sigma.
    \end{equation}
\end{definition}

Bimultiplications are also called bi-translations in \cite{Pet:ide}, multiplications in \cite{Hoc:coh} or (primarily in functional analysis) multipliers \cite{Hel:mul}, \cite{Bus:dou}. The first condition in \eqref{eq.bimultiplication} can be rewritten as $\overrightarrow\sigma(ab) = \overrightarrow\sigma(a)b$ and thus it means that $\overrightarrow\sigma$ is a right $R$-linear map. Similarly, the second condition in \eqref{eq.bimultiplication} means that $\overleftarrow\sigma$ is a left $R$-linear map. Bimultiplications form a unital, associative algebra with product given by opposite composition of left $R$-linear maps and composition of the right $R$-linear ones, that is, for $\sigma_1=(\overleftarrow\sigma_1,\overrightarrow\sigma_1)$ and $\sigma_2=(\overleftarrow\sigma_2,\overrightarrow\sigma_2)$, $\sigma_1\sigma_2 = (\overleftarrow\sigma_2\circ\overleftarrow\sigma_1,\overrightarrow\sigma_1\circ\overrightarrow\sigma_2)$. This algebra is known as the \emph{multiplier algebra of $A$} and is denoted by $M(A)$. A double homothetism is also known as a {\em self-permutable multiplication} \cite{Mac:ext}.

For any $a\in R$, the multiplication from the left and right by $a$ forms a double homothetism, denoted by $\bar{a}$,  that is, for all $b\in R$ 
$$\bar{a}b=ab \quad \text{and}\quad  b\bar{a}=ba.
$$
This is known as an {\em inner homothetism}.
\begin{definition}[Homothetic datum]
    A pair $(\sigma, s)$, where $s\in R$ and $\sigma$ is a double homothetism, is called a \textit{homothetic datum} provided that
    \begin{equation}\label{eq.hom.datum}
        \sigma s = s\sigma \quad \text{ and } \quad  \sigma^{2} = \sigma + \bar{s}.
    \end{equation}
\end{definition}
Given an algebra $R$ and a homothetic datum $(\sigma, s)$,  an extension $R(\sigma, s)$ of $\mathbb{F}$ by $R$ is constructed as follows. The algebra $R(\sigma, s)= R\oplus \mathbb{F}$ as an $\mathbb{F}$-module with the multiplication
\begin{equation}\label{eq.mult}
    (a,\xi)(b,\zeta) = (ab +\zeta a\sigma + \xi \sigma b +\xi\zeta s, \xi\zeta),
\end{equation}
for all $a,b\in R$ and $\xi,\zeta \in \mathbb{F}$. The algebra $R(\sigma, s)$ fits into an extension diagram \eqref{extension},
\begin{equation}\label{hom.extension}
  \begin{tikzcd}
      0 \arrow[r] & R \arrow[r,"\iota_{(\sigma,s)}"] & R(\sigma,s) \arrow[r,"\pi_{(\sigma,s)}"] & \mathbb{F}\ar[r] &0,
  \end{tikzcd}  
\end{equation}
with obvious inclusion  $\iota_{(\sigma,s)}: a\mapsto (a,0)$ and surjection $\pi_{(\sigma,s)}:(a,\xi)\mapsto \xi$. As explained in \cite[Theorem~3.6]{AndBrzRyb:ext}, since any sequence of the form \eqref{extension} is split-exact as a sequence of $\mathbb{F}$-modules, it gives rise to a (unique) homothetic datum such that extension \eqref{extension} is equivalent to \eqref{hom.extension}, that is, there exists a ring isomorphism $\Theta: R(\sigma,s) \to S$ rendering commutative the following diagram
$$
\begin{tikzcd}[column sep=huge,row sep=huge]
      R \arrow[r,"\iota_{(\sigma,s)}"]\arrow[d,"\iota"] & R(\sigma, s) 
      \arrow[d,"\pi_{(\sigma,s)}"] \arrow[dl,swap,"\Theta"] \\
      S \ar[r,"\pi"] & \mathbb{F} .
  \end{tikzcd}  
$$
In view of this equivalence from now on we consider {\em homothetic extensions} of $\mathbb{F}$ by $R$ as in \eqref{hom.extension}.

\begin{remark}\label{rem.hom}
  We note in passing that the homothetic datum conditions $(\sigma,s)$  mean that $R(\sigma,s)$ can be interpreted as the algebra $R \oplus \mathbb{F}\sigma$ with the multiplication induced by the $\mathbb{F}$-bilinearity, i.e.,
  $$
  (a+\xi\sigma)(b+\zeta\sigma) = ab +\xi\sigma b + \zeta a\sigma +\xi\zeta\sigma^2 =  ab +\xi\sigma b + \zeta a\sigma +\xi\zeta s +\xi\zeta\sigma .
  $$
  In other words $\sigma$ can be formally treated as an additional generator $\bs$ of $R(\sigma,s)$ satisfying the following rules, for all $a\in R$
  \begin{equation}\label{sigma.rules}
   a\bs = \overleftarrow\sigma(a), \quad \bs a = \overrightarrow\sigma(a), \quad \bs^2 = \bs +s, \quad \bs s = s\bs.  
  \end{equation}
In what follows we will use $a+\xi \bs$ as the more suggestive notation for $(a,\xi)$ in $R(\sigma,s)$ and we will identify $\iota$ with the canonical inclusion $a\mapsto a +0\bs$ of $R$ in $R\oplus  \mathbb{F}\bs$ and $\pi$ with the canonical projection $a+\xi\bs\mapsto \xi$.
\end{remark}

\begin{proposition}\label{prop.alpha.ext}
    Let $R$ be an algebra with a homothetic datum $(\sigma,s)$. Then an algebra endomorphism $\alpha_R$ of $R$ extends to an algebra endomorphism on $R(\sigma,s)$ satisfying the first condition in \eqref{deriv.ext} if and only if there exists $w\in R$, such that for all $a\in R$,
    \begin{rlist}
        \item $\alpha_R(s) - \varsigma s = w^2 +\varsigma w\sigma +\varsigma \sigma w -w$,
        \item $\alpha_R(a\sigma) = \alpha_R(a) (\varsigma\sigma +w)$,
        \item $\alpha_R(\sigma a) = (\varsigma\sigma +w)\alpha_R(a)$,
    \end{rlist}
    where $\varsigma$ is a scalar equal to either 0 or 1. 
\end{proposition}
\begin{proof}
    In view of \eqref{deriv.ext}, there exist $w\in R$ and $\varsigma\in \mathbb{F}$ such that
    \begin{equation}\label{eq.alphaS}
        \alpha_S(a +\xi\bs) = \alpha_R(a)+\xi w +\xi\varsigma\bs,
    \end{equation}
    for all $a\in R$ and $\xi\in\mathbb{F}$. The $\mathbb{F}$-linear endomorphism of $R(\sigma,s)$ is an algebra endomorphism if and only if
    \begin{equation}\label{alpha.comp}
    \begin{aligned}
    \alpha_R(ab) +\zeta \alpha_R(a\sigma) +\xi\alpha_R(\sigma b) &+\xi\zeta(\alpha_R(s) + w) 
    + \xi\zeta\varsigma\bs =\alpha_R(a)\alpha_R(b) +\zeta (\alpha_R(a)w +\varsigma\alpha_R(a)\sigma)\\
    &+\xi (w\alpha_R(b)  +\varsigma\sigma\alpha_R(b)) +\xi\zeta(w^2 +\varsigma w\sigma +\varsigma\sigma w+\varsigma^2s)  
         + \xi\zeta\varsigma^2\bs
        ,
    \end{aligned}
    \end{equation}
    for all $a,b\in R$ and $\xi,\zeta \in \mathbb{F}$.
   The coordinates at $\bs$ are equal for all $\xi,\zeta$ if and only if $\varsigma$ is an idempotent, and hence $\varsigma =0$ or $\varsigma =1$. Comparing the   sides of \eqref{alpha.comp} at $(\xi,\zeta) = (0,1), (1,0)$ one obtains conditions (ii)-(iii), and taking them into account and then comparing the sides at $(\xi,\zeta) = (1,1)$ one obtains (i). The converse is obvious. 
\end{proof}

By Proposition~\ref{prop.alpha.ext} there are two types of extensions, depending on the parameter $\varsigma$. To indicate the parameter taken we will denote the extended endomorphism by $\alpha_S^\varsigma$. Note that the identity endomorphism on $R(\sigma,s)$ is necessarily the extension of the identity endomorphism on $R$ of type 1, while the zero map is the extension (of the zero map) of type 0.

\begin{proposition}\label{prop.deriv.ext}
    Let $R$ be an algebra with a homothetic datum $(\sigma,s)$ and let $\alpha_R$ be an algebra endomorphism  of $R$ and $w\in R$ satisfying conditions of Proposition~\ref{prop.alpha.ext} so that there is a corresponding endomorphism $\alpha_S^\varsigma$ of $R(\sigma,s)$. An $\alpha_R$-skew derivation $\delta_R$ of $R$ extends to an $\alpha_S^\varsigma$-skew derivation $\delta_S^{\mu(\varsigma)}$ of $R(\sigma,s)$ satisfying conditions \eqref{deriv.ext} if and only if there exists $e\in R$ such that, for all $a\in R$,
    \begin{blist}
        \item $\delta_R(s) -\mu(\varsigma)s= e\sigma + \varsigma\sigma e +we+\mu(\varsigma)w\sigma -e$.
        \item $\delta_R(a\sigma) = \delta_R(a)\sigma +\alpha_R(a)(e+\mu(\varsigma)\sigma)$,
        \item $\delta_R(\sigma a) = (w+\varsigma\sigma)\delta_R(a) +(e+\mu(\varsigma)\sigma)a$,
    \end{blist}
    where $\mu(1) =0$ while $\mu(0)$ is any scalar. 
\end{proposition}
\begin{proof}
    Assume that $\delta_S$ is an $\alpha_S^\varsigma$-skew derivation on $R(\sigma,s)$ that satisfies  \eqref{deriv.ext}, so that  $\delta_S(a) = \delta_R(a)$, for all $a\in R$, and hence there exist $e\in R$ and $\mu(\varsigma)\in \mathbb{F}$ such that
    \begin{equation}\label{eq.deltaS}
        \delta_S(a+\xi\bs) = \delta_R(a) +\xi e +\xi\mu(\varsigma)\bs,
    \end{equation}
    for all $a\in R$ and $\xi\in \mathbb{F}$. Take any  $a+\xi\bs, b+\zeta\bs\in R(\sigma, s)$ and first compute,
    \begin{equation}\label{deriv.lhs}
        \begin{aligned}
            \delta_S\left((a+\xi\bs)(b+\zeta\bs)\right) &= \delta_R(ab) +\zeta\delta_R(a\sigma) +\xi \delta_R(\sigma b) + \xi\zeta(\delta_R(s) +e)+\xi\zeta\mu(\varsigma)\bs,
        \end{aligned}
    \end{equation}
    and then 
    \begin{equation}\label{deltaS.cond}
        \begin{aligned}
       \delta_S(a+\xi\bs)(b+\zeta\bs) &+  \alpha_S(a+\xi\bs)\delta_S(b+\zeta\bs) =\delta_R(a)b +\alpha_R(a)\delta_R(b) + \xi\zeta (e+\mu(\varsigma)w)\sigma +\varsigma\sigma e +\mu(\varsigma)s +we)\\
       + &\xi (eb +\mu(\varsigma) \sigma b+ w\delta_R(b) +\varsigma\sigma\delta_R(b)) + \zeta(\delta_R(a)\sigma +\alpha_R(a)(e+\mu(\varsigma)\sigma))+ \xi\zeta\mu(\varsigma)(1+\varsigma)\bs .
        \end{aligned}
   \end{equation}
   The map $\delta_S$ is an $\alpha_S^{\mu(\varsigma)}$-skew derivation if and only if the right hand side of \eqref{deriv.lhs} is equal to the right hand side of \eqref{deltaS.cond} for all $a,b,\xi,\zeta$. The comparison of the coefficients at $\bs$ yields the constraint $\varsigma\mu(\varsigma) = 0$, hence necessarily $\mu(1) = 0$ and there are no restrictions on $\mu(0)$. Once this is taken into account, 
conditions (a)-(c) are obtained by comparing the right hands sides of equations \eqref{deriv.lhs} and \eqref{deltaS.cond} for $a,b=0,(\xi,\zeta)= (1,1)$ and then for $(\xi,\zeta) = (0,1), (1,0)$. The converse is straightforward.
\end{proof}

\begin{definition}\label{def.comp}
    Given a homothetic datum $(\sigma, s)$ on $R$ a skew derivation $(\alpha_R,\delta_R)$ together with $w,e\in R$ satisfying conditions (i)-(iii) in Proposition~\ref{prop.alpha.ext} and (a)-(c) in Proposition~\ref{prop.deriv.ext} is called a {\em $(\sigma,s)$-compatible} skew derivation. We record all the data involved in a quintuple $(\alpha_R,w;\delta_R,e,\mu(\varsigma))$.
\end{definition}

Immediately from the proofs of Proposition~\ref{prop.alpha.ext} and Proposition~\ref{prop.deriv.ext} we obtain the following corollary.
\begin{corollary}\label{cor.deltaS}
    Let $\delta_S$ be an $\alpha_S$-skew derivation on a homothetic extension $R(\sigma,s)$ of an algebra $R$. If  $\alpha_S$ restricts to the map $\alpha_R: R\to R$ and $\delta_S$ restricts to the map $\delta_R:R\to R$, then $\alpha_R$ is an algebra homomorphism, $\delta_R$ is and $\alpha_R$-derivation and there exist $w, e\in R$ and $\varsigma,\mu(\varsigma)\in \mathbb{F}$ such that conditions (i)-(iii) in Proposition~\ref{prop.alpha.ext}  and (a)-(b) in  Proposition~\ref{prop.deriv.ext} are satisfied.
\end{corollary}
\begin{proof}
    Since $\alpha_R(a)=\alpha_S(a)$ and $\delta_R(a)=\delta_S(a)$, for all $a\in R$, it is clear that $\alpha_R$ is an algebra homomorphism and $\delta_R$ is an $\alpha_R$-derivation of $R$. The parameters $w,e,\varsigma,\mu(\varsigma)$ are determined by
    $$
    w+\varsigma\bs = \alpha_S(\bs) \quad \mbox{and}\quad  e+\mu(\varsigma)\bs = \delta_S(\bs),
    $$ 
    so that
    $\alpha_S$ is given by \eqref{eq.alphaS} and $\delta_S$ is given by \eqref{eq.deltaS}, and hence $\eqref{deriv.ext}$ are satisfied. The assertion follows by Proposition~\ref{prop.alpha.ext}  and   Proposition~\ref{prop.deriv.ext}.
\end{proof}

\begin{example}[Homothetic derivations]\label{ex.inner}
    With any element $b+\zeta\bs$ of $R(\sigma,s)$ one can associate an inner $\alpha_S$-skew derivation. If $\alpha_S$ restricts to an endomorphism $\alpha_R$ of $R$, i.e.\ $\alpha_S(a+\xi\bs) = \alpha_R(a)+\xi w+\xi\varsigma\bs$, for some $w\in R$, satisfying conditions (i)--(iii) in Proposition~\ref{prop.alpha.ext}, then the inner derivation comes out as
    $$
    \begin{aligned}
        \delta_S^{b+\zeta\bs}(a+\xi\bs) &= \alpha_S(a+\xi\bs)(b+\zeta\bs) - (b+\zeta\bs)(a+\xi\bs) \\
        & = \alpha_R(a)b - ba +\zeta(\alpha_R(a)\sigma - \sigma a) + \xi (\varsigma\sigma b - b\sigma +wb) +\xi\zeta (w\sigma +(\varsigma-1)s) + \xi\zeta (\varsigma-1)\bs.
    \end{aligned}
    $$
    The restriction of $\delta_S^{b+\zeta\bs}$ to $R$ gives the corresponding derivation of $R$
    \begin{equation}\label{inner.rest}    
    \delta_R^{b+\zeta\bs}(a) = \alpha_R(a)b - ba +\zeta(\alpha_R(a)\sigma - \sigma a).      
    \end{equation}
    In particular $\delta_R^{b}$ is the inner $\alpha_R$-skew derivation, whereas
    \begin{equation}\label{eq.hom.der}
        \delta_R^{\bs}(a) = \alpha_R(a)\sigma - \sigma a,
    \end{equation}
    is not.
     We will refer to the derivation \eqref{eq.hom.der} as a {\em homothetic skew derivation}. Note in passing that $\delta_R^{\bs}(a)$ remains a skew derivation even if $\sigma$ is merely a bi-multiplication.
    
    Note further that
    $$
    \delta_S^{b+\zeta\bs}(\bs) =  \varsigma\sigma b - b\sigma +wb +\zeta (w\sigma +(\varsigma-1)s) + \zeta (\varsigma-1)\bs.
    $$
    Set
    $$
    e = \varsigma\sigma b - b\sigma +wb +\zeta (w\sigma +(\varsigma-1)s), \quad \mu(\varsigma)=\zeta (\varsigma-1).
    $$
    A straightforward though tedious calculation that uses the fact that $w$ satisfies conditions (i)--(iii) in Proposition~\ref{prop.alpha.ext} confirms that $e$ satisfies conditions (a)--(c) in Proposition~\ref{prop.deriv.ext} in respect to the $\alpha_R$-derivation $\delta_R^{b+\zeta\bs}$ of equation \eqref{inner.rest} (thus confirming the general statement of Corollary~\ref{cor.deltaS} in this case). Furthermore, $\mu(1)=0$, $\mu(0) = -\zeta$. Therefore, the inner $\alpha_S$-derivation $\delta_S^{b+\zeta\bs}$ extends the $\alpha_R$-derivation $\delta_R^{b+\zeta\bs}$. 
    \end{example}
\begin{example}[Extension of rings with zero multiplication]\label{ex.zero}
    Let $R$ be a ring  with the zero multiplication. As shown in \cite[Theorem~6.5]{AndBrzRyb:ext}, up to (weak) equivalence, homothetic ring extensions of $R$ are in one-to-one correspondence with ordered four-fold direct sum decompositions of the abelian group $R = A_1\oplus A_2\oplus A_3\oplus A_4$. Let $\pi_i: R\to A_i$ and $\omega_i:A_i\to R$ be the orthogonal surjections and injections corresponding to the direct sum decomposition, i.e. 
    \begin{equation}\label{orth}
        \pi_i\omega_j = \begin{cases}
        \mathrm{id}:A_i\to A_i, & i=j\,,\cr 0 & i\neq j\, .
    \end{cases}
    \end{equation}
    For an element $a\in R$ we will write $a=a_1+a_2+a_3+a_4$ with $a_i\in A_i$. As explained in \cite[Remark~6.4]{AndBrzRyb:ext} with no loss of generality one can choose the homothetic datum $(\sigma, s)$ corresponding to the decomposition $(A_1,A_2,A_3,A_4)$ with $s=0$ and
    \begin{equation}\label{eq.sigma.zero}
        \overrightarrow\sigma  = \omega_1\pi_1 + \omega_3\pi_3, \qquad \overleftarrow\sigma = \omega_2\pi_2+\omega_3\pi_3.
    \end{equation}
    With this choice the multiplication in $R(\sigma, 0)=R\oplus \mathbb{Z}$ comes out as
    $$
    (a+k\bs)(b+l\bs) = k b_1 + la_2 + la_3+kb_3+ kl\bs.
    $$
    Any additive endomorphism $\alpha$ of $R$ is both a ring endomorphism and a skew-derivation and it is fully described by sixteen additive maps $\alpha_{ij} = \pi_j\alpha\omega_i: A_i \to A_j$, $i,j=1,2,3,4$. In other words $\alpha = \sum_{i,j} \omega_j\alpha_{ij}\pi_i$.

    Condition (i) in Proposition~\ref{prop.alpha.ext} is equivalent to 
    $$ (1-\varsigma)w_1 +(1-\varsigma)w_2 +(1-2\varsigma)w_3 + w_4 = 0,
    $$
    which can be resolved as $w=\varsigma(w_1+w_2)$. Properties (ii) and (iii) yield
    $$
    \begin{aligned}
        \alpha\omega_2\pi_2 +\alpha\omega_3\pi_3 &= \varsigma \left(\omega_2\pi_2\alpha +\omega_3\pi_3\alpha\right) \quad\mbox{and}\quad
        \alpha\omega_1\pi_1 +\alpha\omega_3\pi_3 &= \varsigma \left(\omega_1\pi_1\alpha +\omega_3\pi_3\alpha\right).
    \end{aligned}
    $$
    In view of the orthogonality conditions \eqref{orth} these yield
    $$
    \delta_{2i}\alpha_{2j}+\delta_{3i}\alpha_{3j} = \varsigma\left(\delta_{2j}\alpha_{i2}+\delta_{3j}\alpha_{i3}\right) \quad \mbox{and} \quad \delta_{1i}\alpha_{1j}+\delta_{3i}\alpha_{3j} = \varsigma\left(\delta_{1j}\alpha_{i1}+\delta_{3j}\alpha_{i3}\right) ,
    $$
    for all $i,j=1,\ldots, 4$.
      Consequently, 
$$
\alpha = \varsigma\sum_{i=1}^4\omega_i\alpha_{ii}\pi_i + (1-\varsigma)\sum_{j=1}^4\omega_j\alpha_{4j}\pi_4,
$$
or in components, writing $a_{ij}\in A_{j}$ for $\alpha_{ij}(a_i)$, $a_i\in A_i$,
$$
\alpha(a_1,a_2,a_3,a_4) = (\varsigma a_{11} +(1-\varsigma)a_{41},\varsigma a_{22} +(1-\varsigma)a_{42},\varsigma a_{33} +(1-\varsigma)a_{43}, a_{44}).
$$
These are necessary and sufficient conditions for a general endomorphism $R$ to admit an extension to a ring endomorphism $\alpha_S^\varsigma$ of $R(\sigma,0)$, and, for all $a\in R$ and $k\in \mathbb{Z}$,
$$
\begin{aligned}
    \alpha_S^\varsigma(a+k\bs) &= 
\varsigma\sum_{i=1}^4\omega_i\alpha_{ii}\pi_i(a) + (1-\varsigma)\sum_{j=1}^4\omega_j\alpha_{4j}\pi_4(a)
+ k\varsigma(w_1+w_2)+ k\varsigma\bs\\
&=(\varsigma (a_{11}+kw_1) +(1-\varsigma)a_{41},\varsigma (a_{22}+kw_2) +(1-\varsigma)a_{42},\varsigma a_{33} +(1-\varsigma)a_{43}, a_{44})+ k\varsigma\bs.
\end{aligned}
$$

If $\delta$ is an $\alpha$-skew derivation of $R$ (i.e.\ any additive endomorphism of $R$) we write $d_{ij} = \pi_j\delta\omega_i :A_i\to A_j$. The condition (a) in Proposition~\ref{prop.deriv.ext} yields $e= \varsigma e_1 +e_2 +(1-\varsigma) e_3$.  Conditions (b), (c)  come out as
$$
\begin{aligned}
        \delta\omega_2\pi_2 +\delta\omega_3\pi_3 &=  \omega_2\pi_2\delta +\omega_3\pi_3\delta +\mu(\varsigma)\left(\omega_2\pi_2\alpha +\omega_3\pi_3\alpha\right),\\  
        \delta\omega_1\pi_1 +\delta\omega_3\pi_3 &= \varsigma \left(\omega_1\pi_1\delta +\omega_3\pi_3\delta\right) + \mu(\varsigma) \left(\omega_1\pi_1 +\omega_3\pi_3\right).
    \end{aligned}
$$
These can be resolved in the way similar to (ii)-(iii) in Proposition~\ref{prop.alpha.ext} to give
$$
\begin{aligned}
    \delta &= \omega_2d_{22}\pi_2+\omega_4d_{44}\pi_4+ \varsigma(\omega_1d_{11}\pi_1+\omega_3d_{33}\pi_3)\\
    &+(1-\varsigma)\left(\omega_3d_{23}\pi_2 + \omega_1d_{41}\pi_4\right)+\mu(\varsigma)(\omega_1\pi_1 +\omega_3\pi_3-\omega_2\alpha_{42}\pi_4 - \omega_3\alpha_{43}\pi_4).
\end{aligned}
$$
Thus, explicitly, if we write $a^\delta_{ij} = d_{ij}(a_i)$, the extended derivation in components reads
$$
\delta_S^{\mu(\varsigma)}\left(
\begin{pmatrix}
    a_1\cr a_2\cr a_3\cr a_4
\end{pmatrix} +k\bs\right) = \begin{pmatrix}
    \varsigma (a^\delta_{11}  +ke_1) +(1-\varsigma)a^\delta_{41}+\mu(\varsigma)a_{1} \cr
    a^\delta_{22}-\mu(\varsigma)a_{42} + k e_2 \cr
    \varsigma a^\delta_{33} + (1-\varsigma)(a^\delta_{23} +ke_3) + \mu(\varsigma)(a_3-a_{43} )\cr
    a^\delta_{44}
\end{pmatrix}
+k\mu(\varsigma)\bs.
$$
\end{example}
\begin{example}[Non-degenerate algebras $\daleth_n$] An algebra $A$ is said to be \emph{non-degenerate} if, for all $a\in A$,  $aA =0$ implies $a=0$ and $Aa=0$ implies that $a=0$. In other words, an algebra $A$ is  non-degenerate if and only if it has the zero annihilator, i.e.\
$$
\mathfrak{a}(A) := \{a\in A \;|\; ab=ba=0, \; \mbox{for all $b\in A$}\} = 0.
$$
Non-degenerate algebras are particularly important in the $C^*$-algebra theory. Note that in the case of a non-degenerate algebra, the  map
$$
A\longrightarrow M(A), \qquad a \longmapsto \bar{a},
$$
is a monomorphism of algebras. The image $\bar A$ of this map is an ideal in $M(A)$, the quotient algebra $M(A)/{\bar A}$ is known as the \emph{corona algebra} of $A$.

As shown in \cite[Section~7]{AndBrzRyb:ext}, every multiplier of a non-degenerate algebra  is automatically a double homothetism and up to  equivalence homothetic extensions of $A$ are in one-to-one correspondence with idempotents in the corona algebra of $A$.

The following example is based on the description of the multiplier algebra in \cite[Example~1.20]{VanVer:mul}. Let $\daleth_n$ denote the algebra generated by $n\times n$ unit matrices $e_{1i}, e_{in}$, $i=1,\ldots, n-1$, that is the subalgebra of $n\times n$-matrices with entries from $\mathbb{F}$ of the form
\begin{equation}\label{eq.elements}
  x = \sum_{i=1}^n{\xi_i}e_{1i} + \sum_{i=2}^n{\xi^i}e_{in}=
\begin{pmatrix}
    \xi_1 & \xi_2 & \ldots & \xi_{n-1} & \xi_n \\
    0 & 0& \ldots &  0 & \xi^2\\
    \ldots & \ldots & \ldots &  \ldots & \ldots \\
    0 & 0 &\ldots & 0 & \xi^n
\end{pmatrix}
,
\end{equation}
where $\xi_i,\xi^i\in \mathbb{F}$. This algebra is non-degenerate, since if
$$
0 = e_{11} x =\sum_{i=1}^n{\xi_i}e_{1i}, 
$$
then $\xi_i=0$, for all $i=1,\ldots, n$ and if 
$
0 = e_{1i}x = \xi^i e_{1n},
$ then $\xi^i=0$. Thus the algebra $\daleth_n$  is left non-degenerate. The right non-degeneracy is shown by similar arguments. 

Since $\daleth_n$ satisfies assumptions of \cite[Theorem~7.2]{AndBrzRyb:ext},  equivalence classes of homothetic extensions of $\daleth_n$ are determined by its idempotent multipliers. For any $k=2,\ldots, n-1$, set $\varepsilon_k$ to be the multiplications by $e_{kk}$. Hence, for all $x$ as in equation \eqref{eq.elements},
$$
\varepsilon_kx = \xi^k e_{kn}, \qquad x\varepsilon_k = \xi_ke_{1k}.
$$
Obviously, $\varepsilon_k$ are idempotent multipliers which give non-zero and mutually different elements of the corona algebra of $\daleth_n$. One easily checks that the map
$$
x+\xi\be_k \mapsto x + \xi e_{kk},
$$
is an algebra isomorphism of $\daleth_n(\varepsilon_k, 0)$ and $\daleth_n \oplus \mathbb{F}e_{kk}\subset M_n(\mathbb{F})$, i.e.\ the subalgebra of $M_n(\mathbb{F})$ generated by $e_{1i}, e_{in}$, $i=1,\ldots, n$ and $e_{kk}$. In particular, for $n=3$,  $\daleth_3(\varepsilon_2, 0)$ is (isomorphic to) the multiplier algebra of $\daleth_3$, as explained in \cite[Example~1.20]{VanVer:mul}. Since $\varepsilon_2$ is the unique noninner double homothetism of $\daleth_3$, all homothetic extensions of $\daleth_3$ are equivalent to $\daleth_3(\varepsilon_2, 0)$. For $n>3$ the homothetic extensions $\daleth_n(\varepsilon_k, 0)$ are strict subalgebras of $M(\daleth_n)$.
\end{example}

\begin{example}[Extensions of the zero endomorphism to $\daleth_n(\varepsilon_k, 0)$]\label{ex.ext.0.mor}
    
As an illustration of Proposition~\ref{prop.alpha.ext} we can consider extensions of the zero endomorphism on $\daleth_n$ to homomorphisms $\theta^\varsigma$ of $\daleth_n(\varepsilon_k, 0)$.  The conditions (ii)--(iii) are automatically satisfied, so all extensions are determined by condition (i) (with $s=0$). In the case $\varsigma =0$, condition (i) simply means that $w$ is an idempotent, and hence
\begin{equation}\label{zero.ext.0}
    \theta^0(e_{kk}) = pe_{11} + qe_{nn}, \qquad (p,q)\in \{(0,0),(0,1),(1,0), (1,1)\}.
\end{equation}

In the case $\varsigma=1$ (and hence $\mu(\varsigma)=0$), setting 
$$
w = \sum_{i=1}^n{\upsilon_i}e_{1i} + \sum_{i=2}^n{\upsilon^i}e_{in},
$$
and comparing at the elements of the basis for $\daleth_n$ (the unit matrices) the condition (i) yields
\begin{equation}\label{daleth.cond}
    \upsilon_i = \upsilon_i \upsilon_1, \quad \upsilon^j = \upsilon^j \upsilon^n, \quad \upsilon_k\upsilon_1 = \upsilon^k\upsilon^n = 0, \quad \upsilon_n = \upsilon_1\upsilon_n + \sum_{j=2}^n \upsilon_j\upsilon^j,
\end{equation} 
for all $i=1,\ldots,k,k+1,\ldots, n-1$ and $j =2,\ldots,k,k+1,\ldots, n$. In particular the first two equalities in \eqref{daleth.cond} imply that $\upsilon_1= 0,1$ and $\upsilon^n = 0,1$. Thus also for $\varsigma=1$ there are four possible families of endomorphisms of $\daleth_n(\varepsilon_k,0)$ that restrict to 0 on $\daleth_n$. These come out as:
\begin{subequations}\label{fam.daleth}
    \begin{equation}\label{fam.daleth.1}
        \theta^1(e_{kk}) =  \upsilon e_{1k}+ e_{kk},
    \end{equation}
    \begin{equation}\label{fam.daleth.2}
        \theta^1( e_{kk}) = e_{11} + \sum_{j=2}^n\upsilon_j e_{1j} + e_{kk},
    \end{equation}
    \begin{equation}\label{fam.daleth.3}
        \theta^1( e_{kk}) =  \upsilon_n e_{1n}+\sum_{j=2}^{n-1}\upsilon^j e_{jn} + e_{kk},
    \end{equation}
    \begin{equation}\label{fam.daleth.4}
        \theta^1( e_{kk}) = e_{11} + \sum_{j=2,\, j\neq k}^{n-1}(\upsilon_j e_{1j} - \upsilon_j\upsilon^je_{1n} + \upsilon^j e_{jn}) +e_{nn}+ e_{kk}.
    \end{equation}
\end{subequations}
Of course, $\theta^\varsigma(e_{1i}) = \theta^\varsigma(e_{in}) =0$, for all $i=1,\ldots, n$.
\end{example}

\begin{example}[Extensions of right $\daleth_n$-linear endomorphisms to $\theta^0$-derivations on $\daleth_n(\varepsilon_k, 0)$]
    Since the right $R$-module endomorphisms of $R$ are the same as the $0$-skew derivations of $R$, Proposition~\ref{prop.deriv.ext} can be applied to compute their extensions to $S$. One easily finds that any right $\daleth_n$-module endomorphism of $\daleth_n$ is necessarily of the form
    \begin{equation}\label{r-lin}
       \delta_R(e_{1i}) = \gamma_1e_{1i}, \qquad \delta_R(e_{jn}) = \gamma_je_{jn},
    \end{equation}
    for all $i=1,\ldots n$, $j=2,\ldots,n$, and $\gamma_i\in \mathbb{F}$.

    Let $w=pe_{11}+qe_{nn}$ as in \eqref{zero.ext.0} in Example~\ref{ex.ext.0.mor}. Since $\varsigma =0$, $s=0$, and $k\neq 1,n$, the condition (a) in Proposition~\ref{prop.deriv.ext} reduces to
    $$
    e = ee_{kk}+(pe_{11}+qe_{nn})e, 
    $$
    which can be solved to give
    \begin{equation}\label{e.zero.0}
        e = p\sum_{i=1, i\neq k}^n\nu_i e_{1i}+(1+p)\nu_k e_{1k} + q\nu^ne_{nn},
    \end{equation}
    where $e= \sum_i^n \nu_ie_{1i} + \sum_{j=2}^n\nu^j e_{jn}$. The parameters $\nu_i, \nu^n$ are determined from condition (c) (the condition (b) is automatically satisfied for $\delta_R$ in \eqref{r-lin}). Writing $\mu:=\mu(0)$, one easily finds that
    $$
    p\nu_i =
    \begin{cases} 
    -p\gamma_1 - \mu & i=1,\\
-\mu & i\in \{2,\ldots ,n-1\},\; i\neq k\\
0 & i=n
    \end{cases}, \quad (1-p)\nu_k = \gamma_k -\mu, \quad q\nu^n = -q\gamma_n -\mu .
    $$
    Note, in particular that $\mu =0$ whenever $pq=0$. Inserting these equalitites in \eqref{e.zero.0} we thus find
    $$
    \begin{aligned}
        \delta_S^\mu(a+\xi\bs) = \delta_R(a) &+ \xi\begin{cases}
        -p\gamma_1e_{11} + \gamma_ke_{1n} - q\gamma_n e_{nn} & \mbox{if}\; pq=0,\cr
        -(\gamma_1+\mu)e_{11} - \mu\sum_{j=2,j\neq k}^{n-1}e_{1j} + (\gamma_k-\mu)e_{1k} - (\gamma_n+\mu)e_{nn} & p=q=1 
    \end{cases}\\
    &+ \mu\xi\bs.
    \end{aligned}
    $$
\end{example}

\section{Ore vs homothetic extensions}\label{sec.Ore}

Although the theory of Ore extensions of unital rings has been extensively studied and it is by now very well understood (see e.g.\ \cite{MacRob:Noe}, \cite{GooWar:non}), Ore extensions of nonunital rings are far less explored. To the best of our knowledge the first comprehensive studies of such extensions have been attempted in \cite{BacRicSil:hom} and \cite{LunOinRic:non}.  

Let $\delta_R$ be an $\alpha_R$-skew derivation on an $\mathbb{F}$-algebra $R$. Then a free left $R$-module 
\begin{equation}
  R\oplus \bigoplus_{i\in \mathbb{N}} Rx^i = \{a_0 + \sum_{i=1}^n a_ix^i |n\geq 1, a_i \in R\},
\end{equation}
can be made into an associative algebra in which the canonical inclusion $R\hookrightarrow R\oplus \bigoplus_{i\in \mathbb{N}} Rx^i$, $a\mapsto a\oplus \bigoplus_{i\in \mathbb{N}} 0x^i$, is an algebra monomorphism and with the following multiplication, for all $a,b\in R$, $n,m>0$ 
\begin{equation}\label{eq.Ore.mult}
    a(bx^n) = ab x^n, \quad (ax^m)(bx^n) = \sum_{i=0}^m a\Gamma
[\alpha_R,\delta_R]^m_i(b)x^{n+i},  \quad (ax^m)b = \sum_{i=0}^m a\Gamma
[\alpha_R,\delta_R]^m_i(b)x^{i},
\end{equation}
 where $\Gamma
[\alpha_R,\delta_R]^m_i: R\to R$ denotes the linear map obtained as the sum of all possible compositions of $i$ copies of $\alpha_R$ with $m-i$ copies of $\delta_R$. The resulting algebra is known as  an \textit{Ore extension} of $R$ and it is denoted by $R[x; \alpha_R,\delta_R]$.

A few remarks seem to be in order now. First, note that since $R$ is not necessarily a unital algebra, the generators $x^i$ are not necessarily elements of $R[x; \alpha_R,\delta_R]$. Nethertheless the notation $x^i$ is suggestive of taking powers of $x$ in $R[x; \alpha_R,\delta_R]$, and taking this interpretation of $x^i$ seriously while keeping the aforementioned caveat in mind, the two latter rules in \eqref{eq.Ore.mult} can be deduced from the single rule
$$
axb = a\alpha_R(b)x + a\delta_R(b).
$$
Second, although $x$ in not necessarily an element of $R[x; \alpha_R,\delta_R]$, the multiplications by $x$ from left and right are well defined operations which combine into a double homothetism 
$$
\begin{aligned}
   ax^n\bar x &= ax^{n+1}, \qquad \bar xax^n = xax^n = \alpha_R(a)x^{n+1} + \delta_R(a)x^n,\\
   a\bar x &= ax, \qquad \bar xa = xa = \alpha_R(a)x + \delta_R(a),
\end{aligned}
$$
for all $a\in R$. Third and finally, to make the notation more intuitive we will write $ax^0$ for $a$.

\begin{theorem}\label{thm.data.ext}
    Let $(\sigma, s)$ be a homothetic datum on an algebra $R$ and let $(\alpha_R,w;\delta_R, e,\mu(1))$ be a compatible skew derivation of type 1. Then the datum $(\sigma, s)$ extends to the unique homothetic datum $(\tilde\sigma, s)$ on the Ore extension $R[x;\alpha_R,\delta_R]$ such that
    \begin{equation}\label{eq.extended.bim}
        \tilde{\sigma} = \sigma   \quad \mbox{and} \quad \bar x\tilde{\sigma} = (\bar w +\tilde\sigma)\bar x +\bar e,
    \end{equation}
    on $R$. 
\end{theorem}
\begin{proof}
    In the construction of homothetic datum $(\tilde{\sigma}, s)$  we follow closely the proof of \cite[Theorem~3.6]{AndBrzRyb:ext}. Consider linear maps
    \begin{equation}\label{iota.pi}
        \begin{aligned}
        &\iota':  R[x;\alpha_R,\delta_R]  \longrightarrow R({\sigma},s)[x;\alpha_S,\delta_S],  \quad  & ax^n \longmapsto ax^n, \\
        &\pi':  R({\sigma},s)[x;\alpha_S,\delta_S]  \longrightarrow \mathbb{F}[x],  &(a+\zeta\bs)x^n\longmapsto \zeta x^n.
    \end{aligned}
    \end{equation}
    One easily checks that since $\mu(1) = 0$, i.e.\ $\delta_S(\bs)=0$, both are algebra homomorphisms, and clearly $\iota'$ is injective while $\pi'$ is surjective. Furthermore, $\pi'(\sum_n (a_n+\zeta_n\bs)x^n) = 0$ if and only if $\sum_n\zeta_nx^n = 0$, that is all the $\zeta_n$ vanish. This leads to the following short exact sequence of algebras:
    \begin{equation}\label{extension.X}
  \begin{tikzcd}
      0 \arrow[r] & R[x;\alpha_R,\delta_R] \arrow[r,"\iota'"] & R({\sigma},s)[x;\alpha_S,\delta_S] \arrow[r,"\pi'"] & \mathbb{F}[x]\ar[r] &0.
  \end{tikzcd}  
\end{equation}
   The sequence \eqref{extension.X}  splits as the sequence of $\mathbb{F}$-modules with the splitting maps
   $$
   \begin{aligned}
        &p: R({\sigma},s)[x;\alpha_S,\delta_S]   \longrightarrow R[x;\alpha_R,\delta_R],  \quad  & (a+\zeta\bs) x^n \longmapsto ax^n, \\
        &j:  \mathbb{F}[x] \longrightarrow R({\sigma},s)[x;\alpha_S,\delta_S]   ,  &\zeta x^n \longmapsto \zeta\bs x^n .
    \end{aligned}
   $$
  Define a double operator $\tilde{\sigma}$ on $R[x;\alpha_R,\delta_R]$ by declaring 
   $$
   \tilde{\sigma}ax^n = p (\bs\iota'(ax^n)) , \qquad ax^n \tilde{\sigma}= p (\iota'(ax^n)\bs).
   $$
   Arguing as in the proof of \cite[Theorem~3.6]{AndBrzRyb:ext} one finds that since $\bs\in {\pi'}^{-1}(1)$ and \eqref{extension.X} is a split-exact sequence of $\mathbb{F}$-modules, 
   $$
   \iota'(\tilde{\sigma}ax^n) = \bs\iota'(ax^n) , \qquad \iota'(ax^n \tilde{\sigma})= \iota'(ax^n)\bs.
   $$
   Note that in deriving this conclusion, the fact that the sequence \eqref{extension.X} splits at $R({\sigma},s)[x;\alpha_S,\delta_S]$ plays crucial role.
   The injectivity of the algebra map $\iota'$ yields that $\tilde{\sigma}$ is a double homothetism. And since $\bs^2 - \bs = s$, $s$ complements $\tilde{\sigma}$ to the homothetic datum $(\tilde{\sigma},s)$.

   Note that 
   $$
   \tilde{\sigma} ax^n = p (\bs a x^n)= p(\sigma ax^n) = \sigma a x^n,
   $$
   while, 
   $$
   a \tilde{\sigma} = p(a\bs) = p(a\sigma) = a\sigma,
   $$
   so that $\tilde{\sigma}\mid_R = \sigma$. Moreover, using statements (iii) in Proposition~\ref{prop.alpha.ext} and (c) Proposition~\ref{prop.deriv.ext}
   $$
   \begin{aligned}
       \bar x\tilde{\sigma} ax^{n}  &= x(\sigma a)x^n = \alpha_R(\sigma a)x^{n+1} + \delta_R(\sigma a)x^{n}\\
       &= (w+\sigma)\alpha_R(a)x^{n+1} (w+\sigma)\delta_R(a)x^n +eax^n \\
       &= ((w+\sigma) x +e)ax^n = ((\bar w+\sigma) \bar x +\bar e)ax^n. 
   \end{aligned}
   $$
   Finally,
   $$
   \begin{aligned}
       a \bar x\tilde{\sigma} &=p(ax\bs) = p(a(w+\bs)x) + ae = awx + a\sigma x +ae = a((\bar w+\sigma) \bar x +\bar e),
   \end{aligned}
   $$
   so indeed $\tilde{\sigma}$ satisfies properties \eqref{eq.extended.bim}. 
\end{proof}

\begin{remark}\label{rem.sigma}
    Note that, 
    $$
    ax^n\tilde{\sigma} = p(ax^n\bs) = a\sigma x^n + ap(x^n\bs) = a\sigma x^n + ap\left(\sum_{i=0}^n\Gamma[\alpha_S,\delta_S]_i^n\bs x^i\right). 
    $$
    This gives rise to the following formal commutation rule on $R[x;\alpha_R,\delta_R]$,
    \begin{equation}\label{eq.xn.sigma}
        \bar{x}^n\tilde{\sigma} = \tilde{\sigma}\bar{x}^n + \sum_{i=0}^n\bar\Gamma[\alpha_S,\delta_S]_i^n\bar x^i,
    \end{equation}
    where $\bar\Gamma[\alpha_S,\delta_S]_i^n$ denotes the bi-multiplication  induced by the  elements 
    $p\left(\Gamma[\alpha_S,\delta_S]_i^n\bs\right)$. One easily finds that, for example, 
    $$
    \bar\Gamma[\alpha_S,\delta_S]_n^n = \sum_{i=0}^n \alpha_R^i(w) \quad 
    \mbox{and}\quad \bar\Gamma[\alpha_S,\delta_S]_0^n = \delta_R^{n-1}(e).
    $$
\end{remark}

In view of Proposition~\ref{prop.deriv.ext} and Theorem~\ref{thm.data.ext}, given a homothetic datum and a compatible skew derivation, two combinations of two consecutive extensions of $R$ can be considered. One is the Ore extension of a homothetic extension, the other is the homothetic extension of the Ore extension. The relationship between them is described in the following proposition. 
\begin{proposition}\label{prop.Ore-hom}
    Let $(\sigma, s)$ be a homothetic datum on an algebra $R$ and let $(\alpha_R,w;\delta_R, e,\mu(1))$ be a compatible skew derivation of type 1. Let $(\alpha_S, \delta_S)$ be the extension of $(\alpha_R,w;\delta_R, e,\mu(1))$ to $R(\sigma, s)$ as in Proposition~\ref{prop.deriv.ext} and let $(\tilde\sigma, s)$ be the homothetic datum on $R[x;\alpha_R,\delta_R]$ defined in Theorem~\ref{thm.data.ext}. Then, there exists an algebra monomorphism $\varphi: R[x;\alpha_R,\delta_R](\tilde\sigma, s)  \longrightarrow R({\sigma},s)[x;\alpha_S,\delta_S]$ rendering commutative the following diagram of algebra maps with exact rows,
    \begin{equation}\label{diag.Ore-hom}
  \begin{tikzcd}
      0 \arrow[r] & R[x;\alpha_R,\delta_R] \arrow[r,"\tilde\iota"]\arrow[d,equal] & R[x;\alpha_R,\delta_R](\tilde\sigma, s) \arrow[r,"\tilde\pi"] \arrow[hookrightarrow,"\varphi"]{d} & \mathbb{F}\ar[r]\arrow[hookrightarrow]{d} &0\\
      0 \arrow[r] & R[x;\alpha_R,\delta_R] \arrow[r,"\iota'"] &  R({\sigma},s)[x;\alpha_S,\delta_S]\arrow[r,"\pi'"] & \mathbb{F}[x]\ar[r] &0,
  \end{tikzcd}  
\end{equation}
where $\tilde{\iota}, \tilde{\pi}$ are the canonical maps determining homothetic extension, $\iota',\pi'$ are defined in \eqref{iota.pi}, while the last vertical map is the inclusion of scalars.
\end{proposition}
\begin{proof}
    Consider the linear monomorphism
    $$
    \varphi: R[x;\alpha_R,\delta_R](\tilde\sigma, s)  \longrightarrow R({\sigma},s)[x;\alpha_S,\delta_S], \qquad a+\xi\bs\mapsto \iota'(a)+ \xi\bs.
    $$
    That $\varphi$ fits in the commutative diagram \eqref{diag.Ore-hom} is immediate, so only the multiplicativity of $\varphi$ must be checked. Note that $\varphi\mid_{R({\sigma},s)}$ is the identity map. Also note that, for all $m,n\in \mathbb{N}$, $a\in R$ and $\eta \in \mathbb{F}$,
    $$
    \varphi(ax^{m+n}+ \eta\bs) = \varphi(ax^{m})x^n +  \eta\bs.
    $$
    Throughout the proof, $a,b\in R$ and $\eta,\xi\in \mathbb{F}$ are arbitrary.
    
    We will prove the multiplicativity by induction (on the degrees of the monomials), starting with 
    $$
    \begin{aligned}
        \varphi(ax+\eta\bs)\varphi(b+\xi\bs) &= ax(b+\xi\bs) + \eta\bs (b+\xi\bs)\\
        &=a(\alpha_R(b)+ \xi w+\xi\bs)x + a(\delta_R(b)+ \xi e) +\eta\sigma b +\eta\xi s+ \eta\xi\bs\\
        &=(a\alpha_R(b)+ \xi aw+\xi a{\sigma})x +a\delta_R(b)+ \xi ae +\eta\sigma b +\eta\xi s+ \eta\xi\bs\\
        &= a\alpha_R(b)x+ \xi awx+\xi a{\sigma}x +a\delta_R(b)+ \xi ae +\eta\sigma b +\eta\xi s+ \eta\xi\bs\\
        &=a\alpha_R(b)x+ \xi ax\tilde{\sigma} +a\delta_R(b) +\eta\sigma b +\eta\xi s+ \eta\xi\bs =\varphi((ax+\eta\bs)(b+\xi\bs)),
    \end{aligned}
    $$
    where, apart from the definition of $\varphi$, we use the definition of $(\alpha_S,\delta_S)$ and of the multiplication in the Ore extension $R({\sigma},s)[x;\alpha_S,\delta_S]$, and the relations \eqref{eq.extended.bim}. 

    Next, assume inductively that 
    $$
    \varphi(ax+\eta\bs)\varphi(bx^n+\xi\bs) = \varphi((ax+\eta\bs)(bx^n+\xi\bs)).
    $$
    Then, using the rules of multiplication in $R(\sigma,s)$ and $R[x;\alpha_R,\delta_R]$ we obtain
    $$
    \begin{aligned}
       \varphi(ax+\eta\bs)\varphi(bx^{n+1}+\xi\bs) &= \varphi(ax+\eta\bs)(\varphi(bx^n)x +\xi\bs)\\
       &= \varphi((ax+\eta\bs)bx^n)x +\varphi(\xi(ax+\eta\bs))\bs\\
       &= \varphi(a\alpha_R(b)x^{n+1} + a\delta_R(b)x^{n} +\eta\sigma bx^n)x +
       \varphi(\xi ax\tilde\sigma + \eta\xi s+ \eta\xi\bs)\\
       &=\varphi(a\alpha_R(b)x^{n+2} + a\delta_R(b)x^{n+1} +\eta\sigma bx^{n+1} +\xi ax\tilde\sigma +\eta\xi s+ \eta\xi\bs)\\
       &=\varphi(ax bx^{n+1} +\eta\sigma bx^{n+1} +\xi ax\tilde\sigma +\eta\xi s+\eta\xi\bs) = \varphi((ax+\eta\bs)(bx^{n+1}+\xi\bs)).
    \end{aligned} 
    $$
    Thus, $\varphi$ is multiplicative on products of $ax+\eta\bs$ with $bx^n+\xi\bs$, for all $n\in \mathbb{N}$.
    Finally, assume inductively that there is $m\in \mathbb{N}$, such that for all $n\in\mathbb{N}$,
    $$
    \varphi(ax^m+\eta\bs)\varphi(bx^n+\xi\bs) = \varphi((ax^m+\eta\bs)(bx^n+\xi\bs)).
    $$
    Then
    $$
    \begin{aligned}
        \varphi(ax^{m+1}+\eta\bs)&\varphi(bx^n+\xi\bs) = (\varphi(ax^{m})x+\eta\bs) \varphi(bx^n+\xi\bs)\\
        &= \varphi(ax^{m})x(bx^n + \xi\bs)+ \varphi(\eta\bs)(bx^n+\xi\bs)\\
        &= \varphi(ax^{m})(\alpha_R(b)x^{n+1} + \delta_R(b)x^n + (\xi w+\xi\tilde\sigma)x+ \xi e)
        +\varphi(\eta\sigma bx^n+\eta\xi s +\eta\xi\bs)\\
        &=\varphi(ax^{m})\varphi(\alpha_R(b)x^n)x + \varphi(ax^{m})\varphi(\delta_R(b)x^n +\xi wx +\xi e + \xi \tilde\sigma x)
        +\varphi(\eta\sigma bx^n+\eta\xi s+ \eta\xi\bs)\\
        &= \varphi(ax^{m}\alpha_R(b)x^{n+1}+ax^{m}\delta_R(b)x^n +\xi ax^mwx +\xi ax^{m} e + \xi ax^m\tilde{\sigma}x + \eta\sigma bx^n+\eta\xi s +\eta\xi\bs)\\
        &=\varphi(ax^{m+1}bx^{n} + \xi ax^{m+1}\tilde{\sigma} + \eta\sigma bx^n+\eta\xi s+ \eta\xi\bs) = \varphi((ax^{m+1}+\eta\bs)(bx^n+\xi\bs)).
    \end{aligned}
    $$
    Here in the crucial fifth equality the inductive assumption have been used. This completes the proof of the multiplicativity of $\varphi$.
\end{proof}

\begin{remark}\label{rem.zero.case}
    It is not clear how to lift the restriction on $\varsigma$ in Theorem~\ref{thm.data.ext} and Proposition~\ref{prop.Ore-hom}. While the algebra map $\pi':R({\sigma},s)[x;\alpha_S,\delta_S]\to \mathbb{F}[x]$ such that $\pi'\iota'=0$ can be defined in general by setting $\pi'\left((a+\zeta\bs)x^n\right)= \zeta (\varsigma x+\mu(\varsigma))^n$, this map fails to be surjective for $\varsigma=0$. On the other hand and still in the case $\varsigma =0$ the kernel of the corestriction of $\pi'$ to its image $\mathbb{F}$ contains elements that are not in the image of $\iota'$, as $0 = \sum_n\zeta_n \mu(0)^n =\pi'(\sum_n\zeta_n x^n)$ does not necessarily imply that $\zeta_n=0$, for all $n$. Whether or not the results of Theorem~\ref{thm.data.ext} and Proposition~\ref{prop.Ore-hom} can be extended to the case $\varsigma=0$ is left as an open question at this stage.
\end{remark}

\begin{example}
    To illustrate Theorem~\ref{thm.data.ext} we consider extensions of the zero map on $\daleth_n$  to endomorphisms $\theta^1$ of $\daleth_n(\varepsilon_k,0)$ in equations~\eqref{fam.daleth} and extensions of derivations $\delta_R$ in equation \eqref{r-lin}. In the case \eqref{fam.daleth.1} in which $w=ve_{1k}$ condition (a) in Proposition~\ref{prop.deriv.ext} reduces possible elements $e$ to the form $e= u_1 e_{1k} + u_2 e_{kn} + vu_2e_{1n}$. While condition (b) yields no constraints, (c) yields $u_2=0$ and $u_1 = - \gamma_kv$. Hence every map $\delta_R$ and for $\theta^1$ in \eqref{fam.daleth} can be extended to the $\theta^1$-skew derivation $\delta_S$ of  $\daleth_n(\varepsilon_k,0)$ with $e=-\gamma_k\upsilon e_{1k}$, so that $\delta_S(e_{kk}) =-\gamma_k\upsilon e_{1k}$. By Theorem~\ref{thm.data.ext}, the homothety  $\varepsilon_k$ extends to the homothety $\tilde{\varepsilon}_k$ on $\daleth_n[x;\theta^1,\delta_R]$,
    $$
    (e_{1i}x)\tilde{\varepsilon}_k = (\upsilon\delta_{1i}+\delta_{ik}e_{1k})x - \delta_{1i}\gamma_k\upsilon e_{1k}, \quad (e_{jn}x)\tilde{\varepsilon}_k =0, \quad \tilde{\varepsilon}_k(e_{1i}x)=0, \quad \tilde{\varepsilon}_k(e_{jn}x) = \delta_{jk}e_{kn}x,
    $$
    for all $i=1,\ldots, n$ and $j=2,\ldots,n$.

    While there are no restrictions on either $\theta^1$ in \eqref{fam.daleth.1} or $\delta_R$ for admitting an extension to $\delta_S$, the situation is different in the case of $\theta^1$ as in \eqref{fam.daleth.2}. Since $w=e_{11} + \sum_{j=2}^n\upsilon_j e_{1j}$, condition (a) in Proposition~\ref{prop.deriv.ext} forces $e$ to be of the form
    $$
    e = \sum_{i\neq k}u_i e_{1i} + u_ke_{kn},
    $$
    provided $u_k\upsilon_k =0$.
    As before there are no restrictions arising from (b), while (c) implies that $u_1=-\gamma_1$, $u_k=0$,  $u_i = -\gamma_i\upsilon_1 $, $i=2,\dots k-1,k+1,\ldots, n$ and $\gamma_k\upsilon_k=0$. Thus, provided that $\gamma_k\upsilon_k=0$, $\delta_R$ can be extended to $\theta^1$-derivation $\delta_S$ with 
    \begin{equation}\label{delta.ex2}
       \delta_S(e_{kk})=-\gamma_1e_{11} - \sum_{i=2, i\neq k}^n \gamma_i\upsilon_i e_{1i}. 
    \end{equation}
    If $\gamma_k\upsilon_k\neq 0$, no extension exists.  The extended double homothetism corresponding to the derivation \eqref{delta.ex2} comes out as
    $\tilde{\varepsilon}_k$ on $\daleth_n[x;\theta^1,\delta_R]$,
    $$
    \begin{aligned}
        (e_{1i}x)\tilde{\varepsilon}_k &= \left(\delta_{1i}(e_{11} + \sum_{j=2}^n\upsilon_j e_{1j})+\delta_{ik}e_{1k}\right)x - \delta_{1i}(\gamma_1e_{11}+\sum_{j=2, j\neq k}^n\gamma_j\upsilon_j e_{1j}),\\
         (e_{jn}x)\tilde{\varepsilon}_k &=0, \qquad \tilde{\varepsilon}_k(e_{1i}x)=0, \qquad \tilde{\varepsilon}_k(e_{jn}x) = \delta_{jk}e_{kn}x,
    \end{aligned}  
    $$
    for all $i=1,\ldots, n$ and $j=2,\ldots,n$.

    Restrictions also occur in the cases \eqref{fam.daleth.3} and \eqref{fam.daleth.4}. In the former, an extension exists provided $\upsilon_n\gamma_n = \upsilon^j\gamma_n=0$, for all $j\neq k$, and then $\delta_S(e_{kk}) = -\upsilon^k\gamma_ne_{kn}$, which leads to the double homothety $\tilde{\varepsilon}_k$ given by
    $$
    (e_{1i}x)\tilde{\varepsilon}_k = \left(\delta_{1i}\upsilon_ne_{1n} + (1-\delta_{1i})(1-\delta_{in})\upsilon^ie_{1n} +\delta_{ik}e_{1k}\right)x - \delta_{ik}\upsilon^k\gamma_ne_{1n},
    $$
    and by the same formulae as in the preceding cases. 
    
    For $\theta^1$ in \eqref{fam.daleth.4} an extension of $\delta_R$ exists, provided $\gamma_n\upsilon_j\upsilon^j = 0$, for all $j=2,\ldots, k-1,k+1,\ldots n-1$. If these conditions are satisfied, then
    $$
    \delta_S(e_{kk}) = -\left(\gamma_1 e_{11} + \sum_{j=2,j\neq k}^{n-1} (\gamma_j \upsilon_j e_{1j} +\gamma_n \upsilon^j e_{jn}) + \gamma_n e_{nn}\right).
    $$
    Deriving the form of the double homothety $\tilde{\varepsilon}_k$ is left to the reader.
    
    It might be worth noting that, whenever they exist, extensions of the right $\daleth_n$-module map $\delta_R$ \eqref{r-lin} to $\theta^1$-derivations of $\daleth_n(\varepsilon_k, 0)$ are unique in each case.
\end{example}
\section*{Acknowledgements}

The research of Tomasz Brzeziński is partially supported by the National Science Centre, Poland, through the WEAVE-UNISONO grant no.\ 2023/05/Y/ST1/00046.

\end{document}